\documentclass[reqno]{amsart} 
\usepackage[headings]{fullpage}
\usepackage{amsmath} 
\usepackage{paralist}
\usepackage{amssymb} 
\usepackage{amsthm}
\usepackage{amsfonts}
\usepackage[centertags]{amsmath}
\usepackage{fullpage}
\usepackage{comment} 
\usepackage{bm}
\usepackage{thmtools}
\usepackage{thm-restate}
\usepackage{hyperref, float}
\hypersetup{colorlinks=true,citecolor=purple,linkcolor=purple}

\usepackage{tikz}
\usepackage{graphicx}
\usepackage{mathrsfs}
\usepackage[export]{adjustbox} 
\usepackage{xcolor}
\usepackage[margin=1.25in]{geometry}
\usepackage{amscd}
\usetikzlibrary{shapes}
\usetikzlibrary{arrows}

\theoremstyle{definition}
\newtheorem{thm}{Theorem}[section]
\newtheorem*{thm*}{Theorem}
\newtheorem{cor}[thm]{Corollary}
\newtheorem{mydef}[thm]{Definition}
\newtheorem{lem}[thm]{Lemma}
\newtheorem{prop}[thm]{Proposition}

\newtheorem{example}[thm]{Example}
\newtheorem{caution}[thm]{Caution}
\newtheorem{remark}[thm]{Remark}

\newtheorem{notation}[thm]{Notation}

\DeclareMathOperator{\Spec}{Spec}
\DeclareMathOperator{\spec}{Spec}
\DeclareMathOperator{\height}{ht}

\DeclareMathOperator{\mub}{mub}

\DeclareMathOperator{\vp}{\varphi}
\DeclareMathOperator{\Str}{Str}
\DeclareMathOperator{\str}{Str}

\newcommand{\les}[1]{\le_{\str {#1}}}
\newcommand{\less}[1]{<_{\str {#1}}}
\newcommand{\LB}[1]{\mathbf{L}_{#1}}

\newcommand{\GB}[1]{\mathbf{G}_{#1}}

\newcommand{\Lower}[2]{\mathbf{L}_{#1}(#2)} 
\newcommand{\Greaters}[2]{\mathbf{G}^*_{#1}(#2)}
 
\newcommand{\mf}{\mathfrak}
\newcommand{\strX}[1]{(\str X)_{#1}}
\newcommand{\strY}[1]{(\str Y)_{#1}}
\newcommand{\strXF}{\str_{\text{F}} X}
\newcommand{\strYF}{\str_{\text{F}} Y}

\author{Cory H. Colbert}
\thanks{The author was partially supported by a Lenfest grant from Washington and Lee University.}
\email{ccolbert@wlu.edu}

\title{A Structural Invariant on Certain Two-Dimensional Noetherian Partially Ordered Sets} 
\date{}

\begin{document}

\maketitle

\begin{abstract}
If $(X, \le_X)$ is a partially ordered set satisfying certain necessary conditions for $X$ to be order-isomorphic to the spectrum of a Noetherian domain of dimension two, we describe a new poset $(\str X, \le_{\str X})$ that completely determines $X$ up to isomorphism. The order relation $\le_{\str X}$ imposed on $\str X$ is modeled after R. Wiegand's well-known ``P5" condition that can be used to determine when a given partially ordered set $(U, \le_U)$ of a certain type is order-isomorphic to $(\Spec \mathbb Z[x], \subseteq).$ 
\end{abstract}

\section{introduction}

In 1978, R. Wiegand proved in [\cite{ClassifyP5}, Theorem 1] that there is only one partially ordered set $(X, \le_X)$, up to isomorphism, that satisfies the following five properties: 

 \begin{enumerate}
\item[P1] $X$ is countable and has a unique minimal element.
\item[P2] $X$ has dimension 2.
\item[P3] For each element $x$ of height one there are infinitely many elements $y > x.$
\item[P4] For each pair $x, y$ of distinct elements of height one, there are only finitely many elements $z$ such that $z > x$ and $z > y.$
\item[P5] Given a finite set $S$ of height-one elements and a finite set $T$ of maximal elements, there is a height-one element $w$ such that (a) $w < t$ for each $t \in T,$ and (b) if $s \in S$ and $s < x$ and $w < x$ then $x \in T.$ 
\end{enumerate}

In particular, if $R$ and $S$ are commutative Noetherian rings, and both $(\spec R, \subseteq)$ and $(\spec S, \subseteq)$ satisfy Conditions P1, \ldots, P5, then $(\spec R, \subseteq)$ and $(\spec S, \subseteq)$ are isomorphic as partially ordered sets. Theorem 2 in \cite{RZX} asserts that if $k$ is a field, and $A$ is a two-dimensional domain that is finitely generated as a $k$-algebra, then $(\spec A, \subseteq)$ and $(\spec \mathbb Z[x], \subseteq)$ are isomorphic if and only if $k$ is contained in the algebraic closure of a finite field. In particular, $(\spec \mathbb F_p[x,y], \subseteq)$ and $(\spec \mathbb Z[x], \subseteq)$ are isomorphic for all prime numbers $p,$ but $(\spec \mathbb Q[x,y], \subseteq)$ is not isomorphic to $(\spec \mathbb Z[x], \subseteq)$ because $\mathbb Q$ is not contained in the algebraic closure of a finite field.

In the following years, work was done to understand what other Noetherian spectra satisfy (and do not satisfy) the five conditions, and since the spectrum of many commutative Noetherian rings already satisfies the first four conditions, the work is typically reduced to showing that a candidate spectrum satisfies Condition P5 in order to conclude that the spectrum is isomorphic, as a partially ordered set, to $(\spec \mathbb Z[x], \subseteq).$ It was conjectured in \cite{RZX} that if $A$ is a two-dimensional domain finitely generated as a $\mathbb Z$-algebra, then $(\spec A, \subseteq)$ is isomorphic to $(\spec \mathbb Z[x], \subseteq).$ Important advancements have been made in the direction of that conjecture, including the work of A. Saydam and S. Wiegand, who showed in [\cite{SaydamWiegand}, Theorem 1.2] that $(\spec D[x, g_1, \ldots, g_n], \subseteq)$ satisfies Condition P5, where $D$ is the order in an algebraic number field, $x$ is an indeterminate, and $g_1, \ldots, g_n$ are nonzero elements of the quotient field of $D[x].$ 

While it is known that if two partially ordered sets $(X, \le_X)$ and $(Y, \le_Y)$ satisfy all five conditions then they must be isomorphic, little can be said if they do \textit{not} satisfy P5; they could be isomorphic or they could be very different. For instance, while it is known that $(\spec \mathbb Q[x,y], \subseteq)$ does not satisfy Condition P5, there is no known variation of Condition P5 that would completely classify the spectrum of all Noetherian rings isomorphic to $(\spec \mathbb Q[x,y], \subseteq).$ This inspired us to attach to (most) partially ordered sets satisfying Conditions P1 through P4 a new partially ordered set $\str X$ whose order relation is inspired by Condition P5. The set $\str X$ largely consists of elements of the form $(S, T),$ where $S$ is a finite, nonempty subset of the height-one nodes of $X,$ and $T$ is a nonempty subset of the height-two nodes of $X.$ To get a sense of the order relation on $\str X,$ fix a poset $(X, \le_X)$ satisfying Conditions P1 through P4 (for instance, $X = \spec k[x,y]$), take a finite, nonempty, subset $S$ of height-one nodes of $X$ and a finite, nonempty, subset $T$ of height-two nodes of $X.$ If there is a height-one node $w \in X$ satisfying Condition P5 with respect to $S$ and $T,$ then we will say $(S, T) \less{X} (S\cup\{w\}, T).$ While the definition of $\le_{\str X}$ is more relaxed than requiring the existence of a \textit{single} element $w,$ we show, in our first main result, that $\str X$ has the same classifying effect that Condition P5 does in the sense that if $(\str X, \le_{\str X}) \cong (\str Y, \le_{\str Y}),$ then $(X, \le_X) \cong (Y, \le_Y).$ In the context of $(\spec \mathbb Q[x,y], \subseteq),$ this means that if $R$ is a commutative Noetherian domain such that $$(\str \spec R, \le_{\str \spec R}) \cong (\str \spec \mathbb Q[x,y], \le_{\str \spec \mathbb Q[x,y]})$$ then $$(\spec R, \subseteq) \cong (\spec \mathbb Q[x,y], \subseteq).$$

More precisely, if $(X, \le)$ is a poset that satisfies certain necessary conditions in order for there to exist an isomorphism from $(X, \le_X)$ onto $(\spec k[x,y], \subseteq)$ for some field $k,$ we call $(X, \le_X)$ a $J$-poset. In Section 6 (Theorem \ref{maintheorem}), we prove: 

\begin{thm*} If $X$ and $Y$ are $J$-posets, then $X \cong Y$ if and only if $\str X \cong \str Y.$ Specifically, if $\rho: X\to Y$ is an isomorphism, then the map $$\vp: \str X \to \str Y$$ given by $\vp(A, B) = (\rho(A), \rho(B))$ for all $(A, B) \in \str X$ is an isomorphism, and conversely, if $\vp: \str X \to \str Y$ is any isomorphism, then there is an isomorphism $\rho: X \to Y$ such that $$\vp(A, B) = (\rho(A), \rho(B))$$ for all $(A, B) \in \str X.$ \end{thm*}

In Sections 2 and 3, we set up basic definitions and notation for the objects of study. In Section 4, we define $\str X$ and we provide some basic examples. In Section 5, we begin a deeper study of $\str X,$ and we show $\str X$ is an invariant (i.e., $X \cong Y \implies \str X \cong \str Y;$ see Proposition \ref{invariant}). Section 6 is entirely devoted to proving the second part of the main theorem, the proof of which is broken into two major components: constructing the map $\rho: X \to Y$ and showing $\rho$ is an isomorphism (see Theorem \ref{rho}), and then showing $\vp(A, B) = (\rho(A), \rho(B))$ (Theorem \ref{rhoAB}). Lastly, we study a subset $\strXF$ of $\str X,$ which consists only of elements $(A, B) \in \str X$ such that \textit{both} $A$ and $B$ are finite, and we prove, in Theorem \ref{finiterhoAB}, that if $\strXF$ is endowed with the same order relations as $\str X,$ then every appearance of $\str X$ and $\str Y$ in Theorem \ref{maintheorem} can be replaced with $\strXF$ and $\strYF,$ respectively.

\section{Preliminaries}

In this section, we introduce basic notation and definitions that will be used extensively throughout the paper.

Recall that a partially-ordered set $(X, \le_X)$ is a pair of a set $X$ with a binary relation $\le_X$ that is reflexive, antisymmetric, and transitive. Specifically, this means that for all $x, y, z\in X:$ $x \le_X x$ (reflexivity); $x \le_X y$ and $y \le_X x \implies x =y$ (antisymmetry); and $x \le_X y$ and $y \le _X z \implies  x \le_X z$ (transitivity). We will typically refer to a partially ordered set as a ``poset."


\begin{notation} Let $X$ be a poset, and let $A \subseteq X.$
\begin{enumerate} 
\item If $B \subseteq	X,$ we say $B \le_X A$ if and only if for all $b \in B,$ and for all $a \in A,$ we have $b \le_X a.$ If $B = \{b\}$ is a singleton, and $B \le_X A,$ we will write $b \le_X A$ instead of $\{b\} \le_X A.$ A similar convention will apply if $A$ is a singleton. The notations $B <_X A, A \le_X B, A <_X B$ are defined similarly.
\item Define the following quantities: 
\begin{eqnarray*}
\GB{X}(A) &:=& \{x \in X: A \le_X x\} \text{ and } \GB{X}^*(A) := \GB{X}(A) \setminus A.\\
\LB{X}(A) &:=& \{x \in X: x \le_X A\} \text{ and } \LB{X}^*(A) := \Lower{X}{A}\setminus A.\\
\end{eqnarray*}
\item The ``minimal upper bound set" of $A$ is defined to be $$\mub_X A:= \min \GB{X}{A}.$$ 
\end{enumerate}
\end{notation}

\begin{mydef} Let $(X, \le_X), (Y, \le_Y)$ be posets, and let $f: X \to Y$ be a function of sets. We say $f: (X, \le_X) \to (Y, \le_Y)$ is a \textit{poset map} if and only if for all $x, x' \in X$ we have $x \le_X x' \implies f(x) \le_Y f(x').$ If $f: (X, \le_X) \to (Y, \le_Y)$ is a poset map, we say $f$ is an \textit{order-embedding} if and only if for all $x, x' \in X$ we have $f(x) \le_Y f(x') \implies x \le_X x'.$ Lastly, we say a poset map $f$ is an \textit{isomorphism of posets} (or simply isomorphism) if  $f$ is a surjective, order-embedding from $X$ onto $Y.$ If there is an isomorphism from $(X, \le_X)$ onto $(Y, \le_Y),$ we write $(X, \le_X) \cong (Y, \le_Y)$ or simply $X \cong Y$ if the order relations are clear.  \end{mydef} 

\begin{remark}\label{isoinj} If $f$ is an order-embedding and $f(x) = f(x'),$ then $f(x) \le_Y f(x')$ and $f(x') \le_Y f(x),$ so $x \le_X x'$ and $x' \le_X x,$ thus $x = x'.$ In particular, every order-embedding is necessarily injective on the level of sets. \end{remark} 

\begin{mydef} If $X$ is a poset, we define the \textit{dimension} of $X,$ denoted $\dim X,$ to be $$\sup\{n \in \mathbb Z:\text{There are $x_0, \ldots, x_n \in X$ such that } x_0 <_X x_1 <_X \ldots <_X x_n\}.$$ If there is no bound on the length of chains of nodes in $X,$ we define $\dim X = \infty.$  If $x \in X,$ we define the \textit{height} of $x$ in $X$ to be $\height_X x:=\dim \Lower{X}{x}.$ For each integer $i \ge 0,$ we define $X_i \subset X$ to be the nodes of height $i$ in $X.$  \end{mydef} 





 \section{$J$-Posets}

 \begin{mydef} Let $X$ be a countable partially ordered set with a single minimal node. We say $X$ is a $J$-\textit{poset} if and only if the following conditions are met: 
 \begin{enumerate}
 \item[J1] $\dim X = 2$ and $\mub_X A$ is finite for all nonempty $A \subseteq X.$ 
 \item[J2] If $x \in X_1,$ there are infinitely many $z \in X$ such that $z >_X x.$
 \item[J3] If $m \in X_2,$ and $F \subset X_1$ is finite, then there is a finite set $K \subset X_1$ disjoint from $F$ such that $\mub_X K =\{m\}.$ 
 \item[J4] If $T \subset X_2$ is a finite set, then there is $t \in X_1$ such that $t <_X T.$ 
 \end{enumerate}
 \end{mydef}
 
 \begin{prop}\label{specialt} Let $X$ be a $J$-poset. If $S \subset X_1, T \subset X_2$ are nonempty, finite sets, then there is $t \in X_1 \setminus S$ such that $t <_X T.$ In particular, there are infinitely many $w \in X_1$ such that $w <_X T.$ \end{prop}
 
 \begin{proof} We first claim that if $F$ is a nonempty, finite set, then $X_2 \ne \cup_{f \in F} \Greaters{X}{f}.$ Since $X_1$ is infinite, there is $u \in X_1\setminus F.$ By Condition J2, $\Greaters{X}{u}$ is infinite, and by Condition J1, the set $$\cup_{f\in F} \mub_X\{u, f\} \subset X_2$$ is finite, so there is $$m \in \Greaters{X}{u}\setminus \cup_{f \in F} \mub_X \{u, f\}.$$ In particular, $m \in X_2,$ and $m\notin \cup_{f \in F} \Greaters{X}{f},$ because otherwise we would have $m \in \mub_X \{u, f\}$ for some $f \in F.$ 
 
 Let $S, T$ be as in the statement of the proposition. By the work in the previous paragraph, there is $v \in X_2 \setminus \cup_{s \in S} \Greaters{X}{s}.$ Set $T':=T\cup\{v\}.$ By Condition J4, there is $t <_X T'.$ So $t <_X T,$ and $t \notin S$ because $t <_X v$ and no element of $S$ is less than $v$ in $X.$  \end{proof}

 Items J2 and J4 may at first appear rather restrictive, but many Noetherian spectra satisfy both of those properties. For instance, if $k$ is any field, then the following result of S. McAdam in \cite{MCADAM} implies that $(\spec k[x,y], \subseteq)$ satisfies J4, so is a $J$-poset: 
 
 \begin{thm}\label{McAdam}[\cite{MCADAM}, Corollary 11] Let $R$ be a Noetherian domain, and let $x, y$ be indeterminates over $R.$ If $M_1, \ldots, M_m$ are height 2 primes in $R[x,y],$ then there are infinitely many primes of $R$ contained in $M_1 \cap \cdots \cap M_m.$ \end{thm}

 Condition J3 is satisfied by many Noetherian rings as well, as the next proposition shows.
 
 \begin{prop}If $R$ is a Noetherian ring and $P$ is a prime ideal of $R$ such that $\height P \ge 2,$ and $F$ is a finite set of height-one prime ideals of $R,$ then there is a finite set $K$ of prime ideals such that $K$ is disjoint from $F,$ and $\mub_{\spec R} K = \{P\}.$ \end{prop}

\begin{proof} Since $(\spec R, \subseteq) \cong (\spec R/\sqrt{0}, \subseteq),$ we may assume $R$ is reduced.  Let $\mathcal M$ be the set of all finite sums of height-one prime ideals contained in $P$ but not an element of $F.$ Since $\height P \ge 2$ and $R$ is Noetherian and reduced, there is a regular element $p \in P\setminus \cup_{\mf f \in F} \mf f,$ and thus a height-one prime ideal $\mathfrak p$ with $p \in \mathfrak p \subset P$ by the principal ideal theorem. Since $\mathfrak p \notin F,$ we have $\mathcal M \ne \emptyset.$ Since $R$ is Noetherian, $\mathcal M$ has a maximal element $J = \mathfrak q_1 + \ldots + \mathfrak q_j.$ 

We claim $J = P.$ If $J \ne P,$ then there is $x \in P \setminus J.$ If $x \notin \cup_{\mf f \in F} \mf f,$ then $x$ is an element of a height-one prime ideal $\mathfrak q$ contained in $P$ outside $F,$ so $\mathfrak q + J \in \mathcal M$ and therefore $J = \mathfrak q + J$ by maximality, a contradiction. Therefore, $x \in \cup_{\mf f \in F} \mf f.$ Let $y \in \mathfrak q_1\setminus \cup_{\mf f \in F} \mf f.$ Enumerate the prime ideals in $F$ as $\mf f_1, \ldots, \mf f_k, \mf f_{k+1}, \ldots, \mf f_r$ so that $x \in \cap_{i=1}^k \mf f_i$ and $x \notin \cup_{i=k+1}^r \mf f_i.$ For each $1 \le i \le k,$ let $f_i = 1.$ If $k+1 \le i \le r,$ then choose a regular element $f_i \in \mf f_i$ and not in each $\mf f_s$ such that $s \ne i.$ Then $x + f_1 \cdots f_ry \in P \setminus J,$ and if $x + f_1 \cdots f_ry \in \mf f_i$ for some $1 \le i \le k,$ then one of $f_1, \ldots, f_r, y$ must reside in $\mf f_i,$ an impossibility. If $x + f_1 \cdots f_ry \in \mf f_i$ for $k+1 \le i \le r,$ then $f_i \in \mf f_i$ implies $x \in \mf f_i,$ which is contrary to how we ordered the prime ideals in $F.$ So $x + f_1 \cdots f_ry$ is in a height-one prime ideal $\mathfrak q'$ not in $F,$ and therefore $\mathfrak q' + J \in \mathcal M$ so that $\mathfrak q' + J = J$ by maximality, a contradiction. Therefore, $J = P,$ and if $K = \{\mathfrak q_1, \ldots, \mathfrak q_j\},$ the result is proved.\end{proof} 
 
 \section{The Structure Poset}

Throughout the rest of the paper, $X$ and $Y$ will be $J$-posets. If $Z$ is any set, we will use $\mathscr P(Z)$ to denote the power set of $Z.$ 

\begin{mydef} Define $\str X \subset \mathscr P(X_1) \times \mathscr P(X_2)$ to be the set of all pairs $(A, B)$ of a set $A \subset X_1$ with a set $B \subset X_2$ satisfying the following properties:

\begin{enumerate}
\item $A$ and $B$ are both finite and nonempty, or $(A, B) = (x, \Greaters{X}{x})$ for some $x \in X_1,$ and
\item there is $a \in A$ such that $a <_X B.$
\end{enumerate}
\end{mydef}

\begin{mydef} Let $(A, B), (C, D) \in \str X.$ We say $(C, D)$ dominates $(A, B)$ via set $W$ if and only if the following conditions are met:

\begin{enumerate}
\item[E1] $A \subsetneq C$ and $B \supseteq D.$
\item[E2] $W \subseteq C$ is nonempty and $W <_X D.$
\item[E3] If $a \in A$ and $a <_X m$ and $W <_X m,$ then $m \in D.$ 
\end{enumerate}
\end{mydef} 

If $(A, B) \in \str X,$ and $A = \{a\}$ is a singleton, then we will write $(a, B)$ instead of $(\{a\}, B).$ A similar convention holds if $B = \{b\}$ is a singleton.

\begin{mydef}\label{str} Define the following relation $\les{X}$ on $\str X:$ Declare $(A, B) \les{X} (C, D)$ if and only if either $A = C$ and $B = D$ or $(C, D)$ dominates $(A, B)$ via $W$ for some $W \subseteq C.$ \end{mydef}

\begin{prop} $(\str X, \les{X})$ is a poset. \end{prop} 

\begin{proof} Reflexivity and antisymmetry are clear, so we need only show that if $(A_1, B_1) \less{X} (A_2, B_2)$ and $(A_2, B_2) \less{X} (A_3, B_3),$ then $(A_1, B_1) \less{X} (A_3, B_3).$ If $(A_3, B_3)$ dominates $(A_2, B_2)$ via $W,$ we claim it also dominates $(A_1, B_1)$ via $W.$ Indeed, $A_1 \subsetneq A_3$ and $B_1 \supseteq B_3,$ and if $a \in A_1$ and $a <_X m$ and $W <_X m,$ then $a \in A_2$ and $a <_X m$ and $W <_X m,$ so $m \in B_3.$ \end{proof}

We will say $(C, D)$ dominates $(A, B)$ via $W$ and $(A, B) \les{X} (C, D)$ via $W$ interchangeably. 

\begin{example} Suppose $a, b, c \in X_1, d, e \in X_2$ and $\{a, b\} <_X \{d, e\}$ while $c <_X d$ and $c \not <_X e.$ If $A = \{a, b, c\},$ and $D = \{d, e\}$ then $(A, D) \in \str X$ because $\{a, b\} <_X D.$ However, $(c, D) \notin \str X$ because there is no element in $\{c\}$ that is less than \textit{all} of $D.$ Now, $(c, d) \in \str X$ because $c <_X d.$ Note that although $(c, D) \notin \str X,$ we do have $(\{a, c\}, D) \in \str X$ because $a <_X D.$   \end{example} 

\begin{example} Suppose $S = \{s_1, s_2\}$ and $T = \{t_1\},$ where $s_1 <_X t_1,$ but $s_2 \not <_X t_1.$ Suppose there is $w \in X_1$ such that $\mub\{w, s_2\} = \emptyset$ and $\mub\{w, s_1\} = \{t_1\}.$ Then the pair $S, T$ satisfies P5 with respect to $w.$ Indeed, if $m >_X s_1$ and $m >_X w,$ then $m \in \mub_X\{s_1, w\}$ because $\dim X = 2.$ So $m = t_1 \in T.$ If $m >_X s_2$ and $m >_X w,$ then $m$ is vacuously a member of $T$ because there are no such elements by assumption. In particular, $(S\cup\{w\}, T)$ dominates $(S, T)$ via $W = \{w\}.$ Therefore, $$(S, T) \less{X} (S\cup\{w\}, T).$$ \end{example} 

\begin{example}\label{EX} Suppose $(A, B) \in \str X,$ and $b \in B.$ There is $K \subset X_1$ disjoint from $A$ such that $\mub_X K = \{b\}$ by Condition J3. We claim $(A, B) \less{X} (A \cup K, b).$ Indeed, $K <_X b$ because $\mub_X K = \{b\}.$ Also, if $a \in A$ and $K <_X m$ and $a <_X m,$ then $m \in \Greaters{X}{K},$ so $$\{m\} \ge_X \{b\} = \min \GB{X}{(K)}.$$ Since $m \in X_2$ and $\dim X = 2,$ we have 
$m = b.$ Since $A \subsetneq A \cup K,$ we have $(A \cup K, b)$ dominates $(A, b)$ via $K.$  \end{example} 

\begin{caution} We caution the reader that if $(C, D)$ dominates $(A, B)$ via $W,$ it need not be the case that $W$ is disjoint from $A.$ In fact, in many cases, $W$ will be a subset of $A.$ For instance, if $A = \{a, b, c\}$ and $\mub_X A = \{d, e\}:=D,$ then $(a, D) \less{X} (A, D)$ because $(A, D)$ dominates $(a, D)$ via $A.$ Likewise, $(A, D) \less{X} (A \cup\{f\}, D),$ where $f \in X_1\setminus A,$ because $(A \cup \{f\}, D)$ dominates $(A, D)$ via $A$ as well.  \end{caution} 

\section{Basic Facts and Definitions Regarding $\str X.$}

Throughout this section, $X$ and $Y$ will always be $J$-posets.

\begin{lem}\label{AB} Let $(A, B) \in \str X.$ The following items are true.
\begin{enumerate}
\item $\height_{\str X}(a, \Greaters{X}{a}) = 0$ for all $a \in X_1.$ In particular, if $\height (A, B) > 0,$ then $B$ is finite. 
\item If $(C, D) \in \str X,$ and $b \in B \cap D,$ then there is $J \subset X_1$ such that $(J, b) \in \str X.$ and $$(A, B) \less{X} (J, b),\text{ and }(C, D) \less{X} (J, b).$$ 
\item We have $\height_{\str X}(A, B) > 0$ if and only if $B = \mub K$ for some $K \subseteq A.$ 
\end{enumerate}
\end{lem}

\begin{proof} We prove each part individually.

i) If $(S, T) \les{X} (a, \Greaters{X}{a})$, then $S = \{a\}$ because $S \subseteq\{a\}$ and is nonempty. Since $a <_X T,$ and $T \supseteq \Greaters{X}{a},$ we must have $T = \Greaters{X}{a},$ since the latter set consists of \textit{all} elements of $X_2$ that exceed $a.$ Therefore, $\height_{\str X}(a, \Greaters{X}{a}) = 0.$ Consequently, if $B$ is infinite, then by definition of $\str X,$ we have $(A, B) = (a, \Greaters{X}{a})$ for some $a \in X_1,$ so $\height_{\str X} (A, B) = \height_{\str X} (a, \Greaters{X}{a}) = 0,$ as was just shown.

ii) Choose $K \subset X_1$ disjoint from both $A$ and $C$ such that $\mub_X K = \{b\}.$ By the argument in Example \ref{EX}, if we set $J: = K \cup A \cup C,$ then $(J, d)$ dominates $(A, B)$ via $K$ and $(J, d)$ dominates $(C, D)$ via $K.$

iii) Suppose first that $\height_{\str X} (A, B) > 0.$ Let $(A_0, D) \in \str X$ such that $(A_0, D) \less{X} (A, B),$ and let $a \in A_0$ with $a <_X D.$ Note $a <_X B.$ Let $W \subseteq A$ be such that $(A, B)$ extends $(A_0, D)$ via $W.$ If $W = \{a\},$ then $(A, B)$ extends $(A_0, D)$ via the single element $a.$ Since $a \in A_0,$ Condition E3 would imply $B = \Greaters{X}{a},$ which is infinite. Thus $W \ne \{a\}.$ Set $K = W \cup\{a\}.$ Then $K <_X B.$ Suppose $m >_X K.$ Then $m >_X W$ and $m >_X a$ so $m \in B$ by Condition E3. Therefore, $\mub_X K = B.$

Conversely, if $B = \mub_X K$ for some $K\subseteq A,$ then $|K| \ge 2$ because otherwise we would have $B = K \subset X_1.$ Thus, if $k \in K,$ then $(A, B)$ dominates $(k, B)$ via $K$ because $\{k\} \ne A,$ and $\mub_X K = B$ so Condition E3 is satisfied. In particular, $$(k, B) \less{X} (A, B)$$ so $\height_{\str X}(A, B)  > 0.$ \end{proof}

\begin{mydef} If $B \subset X_2$ is finite and nonempty, define $(\Str X)_B$ to be all elements of $\str X$ whose second ordinate is $B.$ That is, $(\str X)_B$ is the set of all elements in $\str X$ of the form $(A, B).$ Let $(\str X)_B$ inherit the partial ordering on $\str X$ so that $((\str X)_B, \le_{\str X})$ is a subposet of $(\str X, \le_{\str X}).$  \end{mydef}

\begin{notation} If $(A, B) \in \strX{B},$ then we will write $\LB{B}(A, B)$ instead of $\Lower{\strX{B}}{A, B},$ and if $B = \{b\}$ is a singleton, we will similarly write $\LB{b}(A, b)$ instead of $\Lower{\strX{b}}{A, b}.$ In a similar vein, we will write $\height_B(A, B)$ instead of $\height_{\Lower{\strX{B}}{A, B}}(A, B),$ and likewise we will write $\height_b(A, b)$ instead of $\height_{\Lower{\strX{b}}{A, b}}(A, b).$ Finally, if $B = \{b\}$ is a singleton, we will write $(\str X)_b$ instead of $(\str X)_{\{b\}}.$ \end{notation}

\begin{prop}\label{invariant} If $\rho: X \to Y$ is an isomorphism of posets, then the map $\vp: \str X \to \str Y$ given by $$\vp(A, B) = (\rho(A), \rho(B))$$ for all $(A, B) \in \str X$ is an isomorphism of posets. Moreover, $\vp((\str X)_B) = (\str X)_{\rho(B)}$ for all finite, nonempty $B \subset X_2.$ \end{prop} 

\begin{proof} Since $\rho$ is an isomorphism, it is also a height-preserving bijection of sets. Moreover, for any subsets $C, C'$ of $X,$ we have $\rho(C) \le_Y \rho(C') \iff C \le_X C'.$ It is a straightforward exercise to check that $(A, B) \les{X} (A', B') \iff (\rho(A), \rho(B)) \les{Y} (\rho(A'), \rho(B'))$ and that $\vp$ satisfies the conclusion of the proposition. \end{proof}

\begin{example} Let $k$ be an algebraically closed field, and let $x, y$ be indeterminates over $k.$ By Theorem \ref{McAdam}, $(\spec k[x,y], \subseteq)$ satisfies Condition J4, so $(\spec k[x,y], \subseteq)$ is a $J$-poset. Let $\mf m$ and $\mf n$ be two maximal ideal of $k[x,y].$ By the Nullstellensatz, $\mf m = \langle x - a_1, y- b_1 \rangle$ and $\mf n = \langle x - a_2, y- b_2\rangle,$ for some $a_1, a_2, b_1, b_2 \in k.$  In particular, there is an automorphism $\sigma: k[x,y] \to k[x,y]$ carrying $\mf m$ onto $\mf n,$ and this induces an isomorphism $\sigma^*: (\spec k[x,y], \subseteq) \to (\spec k[x,y], \subseteq)$ such that $\sigma^*(\mf m) = \mf n.$ By Proposition \ref{invariant}, $(\str \spec k[x,y])_{\mf m} \cong (\str \spec k[x,y])_{\mf n}.$  \end{example}

\begin{mydef} If $(A, B) \in \str X,$ define $$\ell(A, B) := |\{a \in A: a <_X B\}|,$$ and $$\eta(A, B) := |A| - \ell(A, B).$$ \end{mydef}

Note that $\ell(A, B) \ge 1$ for all $(A, B) \in \str X$ because each pair $(A, B)$ is required to have an element $a \in A$ such that $a<_X B$ in order to be a member of $\str X.$ 

\begin{lem}\label{ABB} Let $(A, B) \in \strX{B}$ such that $\height_B(A, B) > 0.$ Then
\begin{enumerate}
\item $|\LB{B}(A, B)| = (2^{\ell(A, B)} - 1)2^{\eta(A, B)},$ and
\item $B = \mub A$ if and only if $|\LB{B}(A, B)|$ is odd.
\end{enumerate}
\end{lem}

\begin{proof} We prove each part individually.

i) If $\height_B(A, B) > 0,$ then $B = \mub_X K$ for some $K \subseteq A$ by Lemma \ref{AB}. Thus, if $C \subseteq A$ such that there is $c \in C$ with $c <_X B,$ we have $(C, B) \les{X} (A, B)$ via $K.$ There are a total of $2^{|A|}$ subsets of $A,$ of which $2^{\eta(A, B)}$ are subsets $D$ of $A$ such that $(D, B) \notin \str X.$ This is because the sets $D$ are precisely those  subsets of $A$ such that there is no $d \in D$ with $d <_X B.$ That is, $$|\LB{B}(A, B)| = 2^{|A|}-2^{\eta(A,B)} = 2^{\ell(A, B) + \eta(A, B)} - 2^{\eta(A, B)} = (2^{\ell(A, B)}-1)2^{\eta(A, B)}.$$ 

ii)  If $B = \mub_X A,$ then $A <_X B.$ So $\eta(A, B) = 0,$ and therefore $|\LB{B}(A, B)|$ is odd. Conversely, if $|\LB{B}(A, B)|$ is odd, then $\eta(A, B) = 0.$ Since $\height_B(A, B) > 0,$ we have $B = \mub_X K$ for some $K \subseteq A.$ Since $A<_X B$ and $B = \mub_X K,$ it follows that $B = \mub_X A$ as well. \end{proof}

\begin{mydef} For each integer $r \ge 1,$ let $(\mathbb I_r, \le_{\mathbb I_r})$ be the poset with exactly $r$ height-zero nodes $u_1, \ldots, u_r$ and exactly one height-one node $m$ exceeding each height-zero node. \end{mydef}

\begin{example}\label{I2} Let $a, b \in X_1,$ and suppose $\mub_X\{a, b\} = B.$ Then $\LB{B}(\{a, b\}, B)$ is completely described by the following figure:

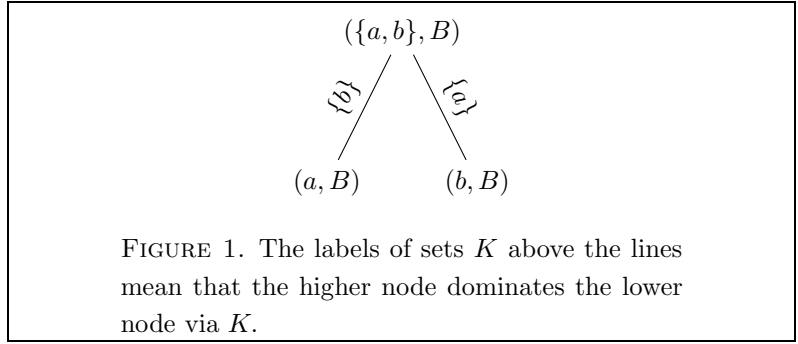
\begin{figure}[h]
\fbox{
\begin{minipage}{10cm}
\centering
\begin{tikzpicture}[scale = 2]
\node (a) at (-.5, 0) {$(a, B)$};
\node (b) at (.5, 0) {$(b, B)$};
\node (c) at (0, 1) {$(\{a, b\}, B)$};
\draw (a) -- (c) node [midway, sloped, above] {$\{b\}$};
\draw (b) -- (c) node [midway, sloped, above] {$\{a\}$};
\end{tikzpicture}
\caption{The labels of sets $K$ above the lines mean that the higher node dominates the lower node via $K.$}
\end{minipage}
}
\label{BmubA}

\end{figure}

The reader will observe that the poset in Figure \ref{BmubA} is order-isomorphic to $\mathbb I_2,$ and this is no coincidence. Indeed, if $\LB{B}(A, B) \cong \mathbb I_2$ for some $(A, B) \in (\str X)_B,$ then $|\LB{B}(A, B)| = 3,$ which is odd, so $B =\mub_X A$ by Lemma \ref{ABB}. Since $|A| \ge 2,$ and $(a, B) \less{X} (A, B)$ via $A$ for all $a \in A,$ it follows that $A = \{a, b\}$ for some distinct $a, b \in X_1.$ That is, $\LB{B}(A, B) \cong \mathbb I_2$ if and only if $A = \{a, b\}$ for distinct $a, b \in X_1$ and $B = \mub_X A.$ 

\end{example}

\begin{example} Recall that in any poset we say $b$ covers $a$ if $a < b$ and no element of the poset is both greater than $a$ and less than $b.$  Let $X$ be a $J$-poset that satisfies Conditions J4 and P5, such as $(\spec \mathbb Z[x], \subseteq).$ Let $u \in X_1$ and $m \in X_2$ such that $u<_X m.$ Then there exists $v \in X_1$ such that $(\{u,v\}, m)$ covers $(u, m)$ in $\str X,$ and $\LB{m}(\{u, v\}, m) \cong \mathbb I_2.$ \end{example}

\begin{example} This example is intended to demonstrate how $\str X$ detects a failure to enjoy Condition P5. Let $k = \overline{\mathbb Q}$ -- where $\overline{\mathbb Q}$ is the algebraic closure of $\mathbb Q$ -- let $x, y$ be indeterminates over $k,$ let $R = k[x,y],$ and let $X = \Spec R.$ Let $P = \langle x^3 - y^2 \rangle R,$ and let $m = \langle x - a, y - b\rangle$ for $a, b \in k$ be a maximal ideal of $R,$ different from $\langle x, y \rangle,$ such that $P \subset m.$  By [\cite{RSEXPO}, Example 2.1], we know there is no height-one prime ideal $Q$ such that $\mub_X \{P, Q\} = \{m\}.$

Now, $(P, m) \in \str X.$ Let $\mu(P, m)$ be the cardinality of the smallest subposet $\Lambda$ of $(\str X)_m,$  containing $(P, m),$ of positive dimension and of the form $\Lower{m}{A, m}$ for some $(A, m) \in \str X.$ If $\Lower{m}{A, m}$ has positive dimension for some $A,$ then $|\Lower{m}{A,m}| \ge 3$ by Lemma \ref{ABB}. Thus, $\mu(P, m) \ge 3.$ If $\mu(P, m) = 3,$ then $\Lambda \cong \mathbb I_2,$ so $\mub_X\{P, Q\} = \{m\}$ for some $Q \in \spec k[x,y]$ by Example \ref{I2}, a contradiction. Therefore, $\mu(P, m) \ge 4.$ 

Of course, $\mu(P, m) \ge 4$ is not unique to this particular case. If $(x, m)$ is any height zero node in $(\str X)_m$, where $x \in X_1,$ then there will fail to exist $y \in X_1$ such that $\mub_X\{x, y\} = \{m\}$ if and only if $\mu(x, m) \ge 4.$ \end{example} 

\section{Proof of Main Theorem} 

We now turn our attention to the main result of this paper, which is that $\str X$ completely determines $X$ up to isomorphism. In other words, we seek to prove that if $\vp: \str X \to \str Y$ is an isomorphism of posets, where $Y$ is a $J$-poset, then there is an isomorphism $\rho: X \to Y$ such that $$\vp(A, B) = (\rho(A), \rho(B))$$ for all $(A, B) \in \str X.$

First, we show that if $m \in X_2,$ then there is a unique $n \in Y_2$ such that $\vp$ restricts to an isomorphism of posets from $(\str X)_m$ onto $(\str Y)_n.$ This fact will demonstrate that there is a bijection of sets $\rho_2: X_2  \to Y_2.$ Having determined where the ``points on the surface" go (i.e., the maximal nodes of $X$ -- thinking of $X$ as the spectrum of a Noetherian ring), we then have a general sense of where the ``irreducible curves on the surface" should go (i.e., the height-one nodes of $X$) since each $x \in X_1$ is determined by $\Greaters{X}{x}$ by Condition J1. In other words, $\rho_2$ will induce a set bijection $\rho_1: X_1 \to Y_1.$ Of course, there is no mystery as to where the ``generic point" (i.e., the unique minimal node of $X$) should go, so the map $\rho_0: X_0 \to Y_0$ is clear. We then set $\rho:=\rho_0\cup \rho_1 \cup \rho_2$ and show that $\rho$ has the desired properties. 

The second part of the argument involves a deeper study of not only the structural information about $X$ contained within $\str X,$ but how well $\vp$ is carrying that information over to $\str Y.$ 

\begin{lem}\label{L1}Let $\vp: \str X \to \str Y$ be an isomorphism of posets. If $\vp(A_1, m_1),$ and  $\vp(A_2, m_2) \in (\str Y)_n$ for some $m_1, m_2 \in X_2, n\in Y_2,$ and sets $A_1, A_2 \subset X_1,$ then $m_1 = m_2.$ In particular, if $\vp(A, M) \in (\str Y)_n,$ then $M = \{m\}$ is a singleton. \end{lem}

\begin{proof} Let $\vp(A_1, m_1) = (C_1, n)$ and let $\vp(A_2, m_2) = (C_2, n).$ By Lemma \ref{AB}, there is $(C_3, n) \in (\str Y)_n$ such that both $$(C_1, n) <_{\str Y} (C_3, n) \text{ and } (C_2, n) <_{\str Y} (C_3, n).$$ Therefore,  both $$(A_1, m_1) <_{\str X} \vp^{-1}(C_3, n):=(A_3, B) \text{ and }(A_2, m_2) <_{\str X} (A_3, B).$$ Now, $B \subseteq \{m_1\} \cap \{m_2\},$ and since $B$ is nonempty, we have $m_1 = m_2.$ 

Suppose $(A, M) \in (\Str X)$ such that $\vp(A, M) = (C, n).$ If $m, m' \in M,$ there are sets $K, K' \subset X_1$ such that both $(A, M) \less{X} (K, m)$ and $(A, M) \less{X} (K', m').$ Therefore, $$(C, n) \less{Y} \vp(K, m) := (J, n)$$ and $(C, n) \less{Y} \vp(K', m') := (J', n).$ Note that we know the second ordinate in the pair $(J, n)$ (resp. $(J', n)$) is $\{n\}$ because it must be a nonempty subset of $\{n\}.$  By the previous paragraph, $m = m'.$ 

\end{proof} 

\begin{lem}\label{L2} If $m \in X_2,$ then the restriction $$\vp|_{(\str X)_m}:(\str X)_m \to \str Y$$ is an isomorphism from $(\str X)_m$ onto $(\str Y)_n$ for some $n \in Y_2.$ \end{lem} 

\begin{proof} Since $\vp: \str X \to \str Y$ is an isomorphism of posets, so is $\vp^{-1}: \str Y \to \str X.$ In particular, Lemma \ref{L1} applied to $\vp^{-1}: \str Y \to \str X$ shows that if $\vp(A, m) = (C, N)$ for some $C \subset Y_1$ and $N \subset Y_2,$ then $N = \{n\}$ is a singleton and thus $\vp((\str X)_m) \subseteq (\str Y)_n$ for some $n \in Y_2.$ 

To see that $\vp((\str X)_m) = (\str Y)_n,$ let $(D, n) \in (\str Y)_n$ and choose any $(L, n) \in (\str Y)_n$ that is in the image of $\vp|_{(\str X)_m}.$ So $(L, n) = \vp(A_1, m)$ for some $(A_1, m) \in (\str X)_m.$ Now, $(D, n) = \vp(A_2, B)$ because $\vp$ is onto, and by Lemma \ref{L1}, $B = \{b\}$ for some $b \in X_2.$

By Lemma \ref{AB}, there is $(J, n)$ in $(\str Y)_n$ exceeding both $(L, n)$ and $(D, n).$ In particular, we have $$(A_1, m)\less{X}\vp^{-1}(J, n) := (A_3, m)\text{ and } (A_2, b) \less{X} (A_3, m)$$ for some $(A_3, m) \in (\str X)_m.$ Therefore, $(A_2, b) \less{X} (A_3, m)$ so $b = m.$ \end{proof}

 Note that $(\str Y)_n \cap (\str Y)_{n'}$ is nonempty if and only if $n = n',$ so the set map $\rho_2 : X_2 \to Y_2$ that sends $m \to n$ as in Lemma \ref{L2} is well-defined. Lemma \ref{L1} shows that $\rho_2$ is injective. If $n \in Y_2,$ then Lemma \ref{L2} applied to $\vp^{-1}|_{(\str Y)_n}: (\str Y)_n \to \str X$ shows that $\rho_2$ is surjective. We now turn our attention to constructing $\rho_1: X_1 \to  Y_1.$

\begin{lem}\label{L3} If $x \in X_1,$ then $\vp(x, \Greaters{X}{x}) = (s, \Greaters{Y}{s})$ for some $s \in Y_1.$ \end{lem} 

\begin{proof} For each $m \in \Greaters{X}{x},$ let $K_m \subset X_1$ be a finite set such that $x \in K_m$ and $\mub_X K_m =\{m\}.$ Then $(x, \Greaters{X}{x}) \less{X} (K_m, m)$ for all $m \in \Greaters{X}{x}.$ By Lemma \ref{L1}, each $\vp(K_m, m) = (J, \rho_2(m))$ for some finite $J \subset Y_1.$ Since $$(x, \Greaters{X}{x}) \less{X} (K_m, m)$$ for all $m \in \Greaters{X}{x},$ and $\vp$ is an isomorphism, the node $\vp(x, \Greaters{X}{x}) \in \str Y$ satisfies $$\vp(x, \Greaters{X}{x}) \less{Y} (J, \rho_2(m))$$ for all such $(J, \rho_2(m)).$ 

Let $(S, T) = \vp(x, \Greaters{X}{x}).$ By the above paragraph, $\rho_2(m) \in T$ for all $m \in \Greaters{X}{x}.$ Since $$\rho_2: X_2 \to Y_2$$ is a bijection of sets, the set $\rho_2(\Greaters{X}{x})$ is infinite and thus $T$ is infinite, so $(S, T) = (s, \Greaters{Y}{s})$ for some $s \in Y_1.$\end{proof}

Define $\rho_1: X_1 \to Y_1$ via $x \to s$ as in Lemma \ref{L3}. $\rho_1$ is well-defined, and if $\rho_1(x_1) = \rho_1(x_2),$ then $\vp(x_1, \Greaters{X}{x_1}) = \vp(x_2, \Greaters{X}{x_2}),$ so $x_1 = x_2$ because $\vp$ is injective (see Remark \ref{isoinj}). That $\rho_1$ is surjective follows after applying Lemma \ref{L3} to $\vp^{-1}: \str Y \to \str X.$  

Let $X_0 = \{x_0\}$ and $Y_0 = \{y_0\}.$ Define $\rho_0: X_0 \to Y_0$ by $\rho_0(x_0) = y_0.$ Set $$\rho:=\rho_0\cup \rho_1 \cup \rho_2.$$

\begin{thm}\label{rho} If $\vp: \str X \to \str Y$ is an isomorphism, then the map $\rho: X \to Y$ as constructed above is an isomorphism of posets. \end{thm} 

\begin{proof} 
By definition, $\rho$ is height-preserving: $\height_Y \rho(x) = \height_X x$ for all $x \in X.$ Additionally, $\rho_i: X_i \to Y_i$ is a bijection of sets for $i = 0, 1, 2.$ In particular, each $\rho_i$ is a surjective and one-to-one map of sets from $X_i$ onto $Y_i$ for $i = 0, 1, 2.$

Suppose $a, b \in X$ and $a \le_X b.$ If $a = x_0,$ the minimal node of $X,$ then $$\rho(a) = \rho(x_0) = y_0 \le Y,$$ so $\rho(a) \le_Y \rho(b)$ in that case. If $a \in X_2,$ then $a = b$ and we have $\rho(a) = \rho(b)$ in that case. We may thus assume $a \in X_1$ and $b \in X_2.$ Since $a <_X b,$ we have $b \in \Greaters{X}{a}.$ Moreover, there is a finite $K \subset X_1$ such that $a \in K$ and $\mub_X K = \{b\}.$ Therefore, $$(a, \Greaters{X}{a}) \less{X} (K, b),$$ and so $$(\rho(a), \Greaters{Y}{\rho(a)}) \less{Y} (J_{\rho(b)}, \rho(b)).$$

Therefore, $\rho(b) \in \Greaters{Y}{\rho(a)},$ or equivalently, $\rho(a) <_Y\rho(b).$

Conversely, assume that $\rho(a) \le_Y \rho(b).$ Similar to above, we may assume that $\rho(a) \in Y_1$ and $\rho(b) \in Y_2$ and $\rho(a) <_Y \rho(b).$  Then there is finite $J \subset Y_2$ such that $\rho(a) \in J$ and $$(\rho(a), \Greaters{Y}{\rho(a)}) \less{Y} (J, \rho(b)).$$ Since $\vp^{-1}(\rho(a), \Greaters{Y}{\rho(a)}) = (a, \Greaters{X}{a})$ and $\vp^{-1}(J, \rho(b)) = (K', b)$ for some $K' \subset X_1,$ we have $a <_X b.$ \end{proof}

We now begin the demonstration that $\vp(A, B) = (\rho(A), \rho(B))$ for all $(A, B) \in \str X.$ 

\begin{mydef}\label{Astar} Let $(A, B) \in \str X,$ and let $\vp, \rho$ be as in Theorem \ref{rho}. Write $\vp(A, B) = (S, T).$ The isomorphism $\rho$ restricts to a bijection of sets from $X_i$ onto $Y_i$ for each $i = 0, 1, 2.$ In particular, since $S \subset Y_1,$ it follows that $S = \rho(A^*)$ for some unique $A^* \subset X_1.$ Likewise, $T = \rho(B^*)$ for some unique $B^* \subset X_2.$ Therefore, given $A\subset X_1, B\subset X_2$ such that $(A, B) \in \str X,$ we define $A^* \subset X_1, B^* \subset X_2$ to be those unique subsets of $X$ such that $\vp(A, B) = (\rho(A^*), \rho(B^*)).$ \end{mydef} 

It is now our goal to show that $A^* = A$ and $B^* = B$ for all $(A, B) \in \str X.$

\begin{lem}\label{FirstStep} Let $(A, B) \in \str X,$ and $\vp, \rho$ as in Theorem \ref{rho}. Write $\vp(A, B) = (\rho(A^*), \rho(B^*))$ for $A^* \subset X_1$ and $B^* \subset X_2$ as in Definition \ref{Astar}. Then the following items hold:
\begin{enumerate}

\item $B^* = B,$ and if $B$ is finite, then the restriction $\vp|_{(\str X)_B}$ is a poset isomorphism from $\strX{B}$ onto $\strY{\rho(B)}.$
\item We have \begin{equation}\label{ELL1} \ell(A^*, B) = \ell(\rho(A^*), \rho(B)) \end{equation} and  \begin{equation}\label{N1}\eta(A^*, B) = \eta(\rho(A^*), \rho(B)) \end{equation}
\item If $\height_B(A, B) > 0,$ then 
\begin{equation}\label{ELL2}  \ell(\rho(A^*), \rho(B)) = \ell(A, B) \end{equation}  and \begin{equation}\label{N2} \eta(\rho(A^*), \rho(B)) =  \eta(A, B).\end{equation} 
\item If $B = \mub A,$ then $A^* = A.$
\end{enumerate}
\end{lem} 
\begin{proof} We prove each item individually.

\textbf{Item 1.} If $b \in B,$ choose $K \subset X_1$ such that $(A, B) \less{X} (K, b).$ Upon applying $\vp,$ we have $$(\rho(A^*), \rho(B^*)) \less{Y} (\rho(K^*), \rho(b)),$$ by Lemma \ref{L2} applied to $(K^*, b) \in (\str X)_b.$ In particular, $\rho(b) \in \rho(B^*),$ so $b \in B^*.$ 

Conversely, if $b \in B^*,$ then there is $F^* \subset X_1$ such that $$(\rho(A^*), \rho(B^*)) \less{Y} (\rho(F^*), \rho(b)).$$ Since $\vp$ is an isomorphism, we must have $$(A, B) \less{X} (F, b)$$ for some $F \subset X_1.$ In particular, $b \in B.$ So $B = B^*.$ 

To see the second part, we need only show that the restriction $\vp|_{\strX{B}}$ is surjective, since the above work shows that $$\vp|_{\strX{B}}: \strX{B} \to \strY{\rho(B)}$$ is well-defined, and the map is an order-embedding because $\vp$ is. Indeed, if $(\rho(C^*), \rho(B)) \in \strY{\rho(B)},$ and $\vp(C, D) = (\rho(C^*), \rho(B)),$ then, by the above paragraph, $\rho(B) = \rho(D),$ so $B = D$ since $\rho$ is a set bijection. 

\textbf{Item 2.} Equations \ref{ELL1} and \ref{N1} follow from the definitions of the functions $\ell, \eta: \str X \to \mathbb Z^{\ge 0}$ and Theorem \ref{rho} since $$U \le_X V \iff \rho(U) \le_Y \rho(V)$$ for all subsets $U, V$ of $X.$ 

\textbf{Item 3.} By Item 1, the restriction $$\vp|_{\strX{B}}: \strX{B} \to \strY{\rho(B)}$$ is an isomorphism of posets. Let $\sigma := \vp|_{\strX{B}}.$ Then, \begin{equation}\label{LB} |\LB{B}(A, B)| = \left|\sigma(\LB{B}(A, B))\right| = |\LB{\rho(B)}(\sigma(A, B))| = |\LB{\rho(B)}(\rho(A^*), \rho(B))|.\end{equation} 

Since $\height_B(A, B) > 0,$ we have $|\LB{B}(A, B)| = (2^{\ell(A, B)} - 1)2^{\eta(A, B)}$ by Lemma \ref{ABB}. Since $\vp$ is an isomorphism, $$\height_{\rho(B)} (\rho(A^*), \rho(B^*)) = \height_{\rho(B)} \vp(A, B) = \height_B(A, B) > 0.$$ Therefore, by Lemma \ref{ABB}, we have $$|\LB{\rho(B)}(\rho(A^*), \rho(B^*))| = (2^{\ell(\rho(A^*), \rho(B^*))}-1)2^{\eta(\rho(A^*), \rho(B^*))}.$$

By Equation \ref{LB} and the Fundamental Theorem of Arithmetic, we have $$\eta(\rho(A^*), \rho(B^*)) = \eta(A, B)$$ and $$\ell(\rho(A^*), \rho(B^*)) = \ell(A, B)$$ as desired.

\textbf{Item 4.} Suppose $B = \mub A.$ As in Item 3, let $\sigma$ be $\vp|_{\strX{B}}.$ By Lemma \ref{ABB} and the fact that $\sigma$ is an isomorphism from $\strX{B}$ onto $\strY{\rho(B)},$ we have $$\height_B(A, B) = \height_{\rho(B)}(\rho(A^*), \rho(B)) > 0.$$ By Items 2 and 3 and the fact that $A <_X B,$ we have \begin{equation}\label{EQ1} \eta(\rho(A^*), \rho(B))= \eta(A^*, B) = \eta(A, B)  = 0\end{equation} and \begin{equation}\label{EQ2} \ell(\rho(A^*), \rho(B)) = \ell(A^*, B) = \ell(A, B). \end{equation} By Equations \ref{EQ1} and \ref{EQ2}, we have $|A^*| = |A|.$ 

If $a \in A,$ then $(a, \Greaters{X}{a}) \less{X} (A, B)$ via $A.$ By Lemma \ref{L3}, $$(\rho(a), \Greaters{Y}{\rho(a)}) \less{Y} (\rho(A^*), \rho(B)).$$ So $\rho(a) \in \rho(A^*),$ and thus $a \in A^*.$ Since $a$ was arbitrary, we have $A \subseteq A^*.$ Finally, since both $A$ and $A^*$ are finite sets with $|A| = |A^*|,$ we have $A = A^*$ as desired. 
\end{proof} 

\begin{remark}\label{Item1comment} A careful inspection of the proof of Item 1 shows that we only use the fact that $\rho_2:X_2 \to Y_2$ is a set bijection of sets. \end{remark} 

\begin{mydef}\label{NCD1} If $(C, D) \in \str X$ and $d \in D,$ let $$\mathcal N(C, d) := \{(A, d) \in \str X: (C, d) \less{X} (A, d) \text{ and } \eta(C, d) = \eta(A, d)\}.$$ \end{mydef}

\begin{remark}\label{NCD} If $(C, D) \in \str X$ and $d \in D,$ then $(C, d) \less{X} (A, d)$ if and only if $(C, D) \less{X} (A, d).$ \end{remark}

\begin{cor}\label{SecondStep} Suppose $(C, D) \in \str X, d \in D$, and let $\mathcal N(C, D)$ be as in Definition \ref{NCD1}. If $\vp(A, d) = (\rho(A), \rho(d))$ for all $(A, d) \in \mathcal N(C, d),$ then $\vp(C, D) = (\rho(C), \rho(d)).$ In particular, if $C <_X d,$ then $\vp(C, D) = (\rho(C), \rho(D)).$ \end{cor}

\begin{proof}  

By Lemma \ref{FirstStep} and Definition \ref{Astar}, we may write $\vp(C, D) = (\rho(C^*), \rho(D)).$ We claim $C = C^*.$ By Condition J3, there exists a finite set $F \subset X_1$ such that $F$ is disjoint from $C\cup C^*$ and $\mub_X F = \{d\}.$ Therefore, \begin{equation}\label{Ineq1} (C, D) \less{X} (C \cup F, d)\end{equation} because $(C \cup F, d)$ dominates $(C, D)$ via $F.$ Now, $$\eta(C \cup F, d) = \eta(C, d)$$ because $F <_X d.$ Therefore, $$(C \cup F, d) \in \mathcal N(C, d),$$ (see Remark \ref{NCD}) and so \begin{equation}\label{CorEq1} \vp(C \cup F, d) = (\rho(C \cup F), \rho(d)).\end{equation} 

 Applying $\vp$ to (\ref{Ineq1}) and using (\ref{CorEq1}), we see that $$(\rho(C^*), \rho(D)) \less{Y} (\rho(C \cup F), \rho(d)).$$ Therefore, $C^* \subseteq C \cup F,$ and since $F$ is disjoint from $C^*,$ we have $C^* \subseteq C.$ 
 
 Now $(\rho(C^*) \cup \rho(F), \rho(d))$ dominates $(\rho(C^*), \rho(D))$ via $\rho(F)$ because $$\rho(\mub_X F) = \mub_X \rho(F) = \{\rho(d)\},$$ which follows from the fact that $\rho$ is an isomorphism. Now, there is $J \subset X_1$ such that $$\vp(J, d) = (\rho(C^* \cup F), \rho(d)).$$ 
 
 We claim $(J, d) \in \mathcal N(C, d).$ Since $$(\rho(C^*), \rho(D)) \less{Y} (\rho(C^*) \cup \rho(F), \rho(d)),$$ we have $$(C, D) \less{X} (J, d).$$ We need only show that $\eta(J, d) = \eta(C, d)$ in order to conclude that $(J, d) \in \mathcal N(C, d).$ Since $C \subseteq J,$ we have $\eta(C, d) \le \eta(J, d).$ By parts 2 and 3 of Lemma \ref{FirstStep} (note that $J^* = C^* \cup F$), \begin{equation}\label{CorEq2}\eta(J, d) = \eta(J^*, d) = \eta(C^* \cup F, d) = \eta(C^*, d) \le \eta(C, d).\end{equation} where the last inequality in (\ref{CorEq2}) follows because $C^* \subseteq C$ by the above work. Therefore, $\eta(C, d) = \eta(J, d),$ so $(J, d) \in \mathcal N(C, d),$ and thus $$\vp(J, d) = (\rho(J), \rho(d)) = (\rho(J^*), \rho(d)) = (\rho(C^* \cup F), \rho(d)).$$ Therefore, $$J = J^* = C^* \cup F.$$ Since $C \subseteq J,$ and $C$ is disjoint from $F,$ we have $C \subseteq C^*$ so $C = C^*.$ 
 
 To see the final part of this corollary, suppose $C <_X d,$ and suppose $(A, d) \in \mathcal N(C, d).$ Then $$\eta(A, d) = \eta(C, d) = 0,$$ and $\height_d(A, d) > 0,$ so $$\mub_X A = \{d\}$$ by Lemma \ref{ABB}. By part 4 of Lemma \ref{FirstStep}, $$\vp(A, d) = (\rho(A), \rho(d)).$$ By the above work, $\vp(C, D) = (\rho(C), \rho(D)).$ \end{proof}

\begin{thm}\label{rhoAB} Let $\vp, \rho$ be as in Theorem \ref{rho}. If $(A, B) \in \str X,$ then $\vp(A, B) = (\rho(A), \rho(B)).$ \end{thm} 
\begin{proof} We claim that it suffices to prove the theorem for all nodes in $\str X$ of the form $(A, b)$ where $\height_b(A, b) > 0.$ Given that assumption, if $(C, B) \in \str X, b \in B,$ and $(A,b) \in \mathcal N(C, b)$ from Definition \ref{NCD1}, then $\height_b(A, b) > 0$ because $(A, b) >_{\str X} (C, b).$ Thus $\vp(A, b) = (\rho(A), \rho(b)),$ and so by Corollary \ref{SecondStep}, we have $\vp(C, B) = (\rho(C), \rho(B)).$ 

Therefore, let $(A, b) \in \str X$ such that $\height_b(A, b) > 0.$ If $\mub_X A = \{b\},$ then the result holds by part 4 of Lemma \ref{FirstStep}. Since $\height_b(A, b) > 0,$ we have $\{b\} = \mub_X A$ if and only if $\eta(A, b) = 0.$ Thus, assume $\eta(A, b) \ne 0,$ and enumerate the elements $q \in A$ such that $q \not < b$ as $q_1, \ldots, q_{\eta(A,b)}.$ For each $i = 1, \ldots, \eta(A, b),$ let $b_i \in X_2$ be a node of $X$ such that $q_i <_X b_i,$ and set $B_i := \{b, b_i\}.$ Choose, by Proposition \ref{specialt}, $x \in X_1$ such that $x <_X \cup_{i=1}^{\eta(A, b)} B_i$ and $x \notin A.$ Define $Q_i:=\{x, q_i\}$ for each $i = 1, \ldots, \eta(A, b).$ 

Recall that for a pair $(A, B),$ we define $A^L = \{a \in A: a <_X B\}.$ Write $A = A^L \cup \{q_1, \ldots, q_{\eta(A, b)}\}.$ Note that $\{b\} = \mub_X A^L,$ and thus \begin{equation}\label{MTeq1} (A, b) \less{X} (A \cup\{x\}, b)\end{equation} and \begin{equation}\label{MTeq2} (Q_i, B_i) \less{X} (A  \cup\{x\}, b)\end{equation} for all $1 \le i \le \eta(A, b).$ 

Let $$\vp(A, b) = (\rho(A^*), \rho(b))$$ and $$\vp(A\cup\{x\}, b) =(\rho((A \cup \{x\})^*), \rho(b)).$$  

We claim $A^* = A.$ In order to prove this, we first show $(A \cup \{x\})^* = A \cup \{x\}$ and conclude that $A = A^*.$

By Lemma \ref{FirstStep}, we have $$\eta(A, b) = \eta(A^*, b) \text{ and } \ell(A, b) = \ell(A^*, b),$$ so $|A| = |A^*|.$ By, our choice of $x \in X_1 \setminus A,$ we have $$\eta(A, b) = \eta((A \cup \{x\})^*, b)\text{ and }\ell((A \cup\{x\})^*, b) = 1 + \ell(A, b),$$ so $|(A \cup \{x\})^*| = 1 + |A|.$ Moreover, by (\ref{MTeq1}), $\rho(A^*) \subset \rho((A \cup \{x\})^*)$ so $A^* \subset (A \cup \{x\})^*.$ 

Since $A^L <_X b$ and $(A^L, b) \less {X} (A \cup\{x\}, b)$ via $A^L,$ Corollary \ref{SecondStep} implies $\vp(A^L, b) = (\rho(A^L), \rho(b)),$ so $$(\rho(A^L), \rho(b)) \less{Y} (\rho((A \cup\{x\})^*), \rho(b)),$$ and thus $A^L \subset (A \cup \{x\})^*.$ 

Fix $1 \le i \le \eta(A, b).$ By the choice of $B_i$ and Corollary \ref{SecondStep}, we have $$\vp(Q_i, B_i) = (\rho(Q_i), \rho(B_i)).$$ After applying $\vp$ to (\ref{MTeq2}) and using the fact that $\rho$ is an isomorphism, we see $Q_i \subset (A \cup \{x\})^*.$ Since $i$ was arbitrary, we have $\cup_{i=1}^{\eta(A, b)} Q_i \subset (A \cup \{x\})^*,$ and by choice of each $Q_i,$ it follows that $$|\cup_{i=1}^{\eta(A, b)} Q_i| = \eta(A, b) + 1.$$

Putting all of this together, we see that we have accounted for $1 + |A|$ elements in $(A \cup \{x\})^*:$ $\ell(A, b)$ elements from $A^L$ and $\eta(A, b)+1$ elements from $\cup_{i=1}^{\eta(A, b)} Q_i.$ Therefore, $$(A\cup\{x\})^* = A^L \cup \{q_1, \ldots, q_{\eta(A, b)}\} \cup \{x\} = A \cup \{x\}.$$

Since $A^* \subset A \cup\{x\},$ it suffices to show $x \notin A^*.$ If $x \in A^*,$ then $\rho(x) \in \rho(A^*),$ so $$(\rho(x), \Greaters{X}{\rho(x)}) = \vp(x, \Greaters{X}{x}) \less{Y} \vp(A, b) = (\rho(A^*), \rho(b)).$$ Since $\vp$ is an isomorphism, we have $$(x, \Greaters{X}{x}) \less{X} (A, b),$$ so $x \in A,$ a contradiction. Therefore, $A^* \subset A,$ and since both $A, A^*$ are finite sets with $|A| = |A^*|,$ we have $A = A^*.$ \end{proof} 

The previous theorem, combined with Proposition \ref{invariant}, immediately gives us:

\begin{thm}\label{maintheorem} If $X$ and $Y$ are $J$-posets, then $X \cong Y$ if and only if $\str X \cong \str Y.$ Specifically, if $\rho: X\to Y$ is an isomorphism, then the map $$\vp: \str X \to \str Y$$ given by $\vp(A, B) = (\rho(A), \rho(B))$ for all $(A, B) \in \str X$ is an isomorphism, and conversely, if $\vp: \str X \to \str Y$ is any isomorphism, then there is an isomorphism $\rho: X \to Y$ such that $$\vp(A, B) = (\rho(A), \rho(B))$$ for all $(A, B) \in \str X.$ \end{thm}

\begin{mydef} Define the following subposet of $\str X:$ $$\strXF := \{(A, B) \in \str X: A, B \text{ are finite}\}.$$ If $B \subset X_2$ is finite, define $(\strXF)_B := (\str X)_B.$ \end{mydef}

As a corollary to Theorem \ref{rhoAB}, we show that an isomorphism from $\psi: \strXF \to \strYF$ induces an isomorphism $\rho: X \to Y$ such that $\psi(A, B) = (\rho(A), \rho(B))$ for all $(A, B) \in \strXF.$ For the rest of the paper, let $\psi: \strXF \to \strYF$ be an isomorphism of posets. By Theorem \ref{rho} and \ref{rhoAB}, it suffices to find an isomorphism $\vp: \str X \to \str Y$ such that $\vp$ restricts to $\psi$ on $\strXF.$ Note that the proofs of Lemmas \ref{ABB}, \ref{L1}, \ref{L2} apply to $\psi$ and $\strXF.$ Also, by Remark \ref{Item1comment}, the proof of Item 1 of Lemma \ref{FirstStep} applies to $\psi$ and $\strXF$ as well. 

\begin{lem} Let $X, Y$ be $J$-posets, let $\psi: \strXF \to \strYF$ be an isomorphism of posets, and let $x \in X_1.$ Let $b_1, b_2, b_3, \ldots$ be an enumeration of $\Greaters{X}{x},$ and, for each $i \ge 1,$ let $\mathcal K_i(x)$ be all $(K, b_i) \in \strXF$ such that $x \in K$ and $\mub_X K = \{b_i\}.$ Let $\mathcal K_{\strXF}(x) = \cup_{i=1} \mathcal K_i(x).$ Then $$\psi(\mathcal K_{\strXF}(x)) = \mathcal K_{\strYF}(y)$$ for some unique $y \in Y_1.$ \end{lem}

\begin{proof} If $Z$ is any $J$-poset, and $z, w \in Z_1$ are distinct elements such that $\mathcal K_{\str_FZ}(z) \subseteq \mathcal K_{\str_ZF}(w),$ then $z$ and $w$ are both less than every element $t \in Z_2$ such that $z <_Z t.$ Since $z, w$ are distinct, we have $\Greaters{Z}{z} \subset \mub_Z\{z, w\},$ a contradiction because $\mub_Z\{z, w\}$ is finite. The uniqueness part of the lemma follows, and it suffices to show that $\psi (\mathcal K_{\strXF}(x)) \subseteq \mathcal K_{\strYF}(y)$ for some $y$ because we may apply the same argument to $\psi^{-1}:\strYF \to \strXF$ to show that $$\mathcal K_{\strXF}(x) \subseteq \psi^{-1}(\mathcal K_{\strYF}(y)) \subseteq \mathcal K_{\strXF}(x')$$ which implies $x = x'$ so $\psi(\mathcal K_{\strXF}(x)) = \mathcal K_{\strYF}(y).$

With the enumeration of $\Greaters{X}{x}$ as $b_1, b_2, \ldots,$ let $B_i := \{b_1, \ldots, b_i\}$ for each integer $i \ge 1.$ By Item 1 of Lemma \ref{FirstStep}, we may write, for each integer $j \ge 1,$ $\psi(x, B_j) = (F_j, \rho_2(B_j))$ for some finite $F_j \subset Y_1$ (see Remark \ref{Item1comment}). Now, $$(F_j, \rho_2(B_j)) <_{\strYF} (J, \rho_2(b_i))$$ for all $i \le j$ such that $(J, \rho_2(b_i)) \in \psi(\mathcal K_i(x)).$ Now,  $$F_j \subseteq \bigcap_{i=1}^j \bigcap_{(J, \rho_2(b_i)) \in \psi(\mathcal K_i(x))} J:=W_j,$$ so each $W_j$ is nonempty. Moreover, $W_j \supseteq W_r$ for all $r \ge j.$ Since each $W_j$ is finite, there exists $j_0$ such that $W_{j_0} = W_r$ for all $r \ge j_0.$ In particular, $$\bigcap_{(J, \rho_2(b_i)) \in \cup_{i=1}^{\infty} \psi(\mathcal K_i(x))} J = \bigcap_{j=1}^{\infty} W_j$$ is nonempty. Let $y \in \cap_{j=1}^{\infty} W_j.$ If $(K, b_i) \in \mathcal K_i(x),$ then $\height_X(K, b_i) > 0$ and $\eta(K, b_i) = 0,$ so $\height \psi(K, b_i) > 0$ and $\eta(\psi(K, b_i)) = 0,$ and therefore $\mub_Y J = \{\rho_2(b_i)\}$ where $(J, \rho_2(b_i)) = \psi(K, b_i).$ Since $y \in \cap_{j=1}^{\infty} W_j,$ we have $y \in J,$ so $\psi(K, b_i) \in \mathcal K_{\strYF}(y).$ Thus, $\psi(\mathcal K_{\strXF}(x)) \subseteq \mathcal K_{\strYF}(y).$ \end{proof}

\begin{thm}\label{finiterhoAB} If $X, Y$ are $J$-posets, then $\strXF \cong \strYF$ if and only if $X \cong Y.$ Additionally, if $\psi: \strXF \to \strYF$ is an isomorphism, then there is an isomorphism $\rho: X\to Y$ such that $$\psi(A, B) = (\rho(A), \rho(B))$$ for all $(A, B) \in \str X.$  \end{thm} 

\begin{proof} If $\rho: X \to Y$ is an isomorphism, then $\psi: \strXF \to \strYF$ given by $\psi(A, B) = (\rho(A), \rho(B))$ is an isomorphism. To get the other direction, it suffices to define a poset isomorphism $\vp: \str X \to \str Y$ that extends $\psi.$ 

If $(A, B) \in \strXF,$ define $\vp(A, B) = \psi(A, B).$ Otherwise, $(A, B) = (a, \Greaters{X}{a})$ for some $a \in X_1.$ Then $\psi(\mathcal K_{\strXF}(a)) = \mathcal K_{\strYF}(c)$ for some unique $c \in Y_1$ by the previous lemma. Define $\vp(a, \Greaters{X}{a}) = (c, \Greaters{X}{c}).$

If $(A_1, B_1) \les{X} (A_2, B_2),$ and $B_1$ is finite, then so is $B_2,$ so $\vp(A_1, B_1) \les{X} \vp(A_2, B_2)$ because $\vp$ extends $\psi.$  If $B_1$ is infinite and $B_2$ is finite, then write $\vp(A_1, B_1) = \vp(a, \Greaters{X}{a}) = (c, \Greaters{X}{c}) ,$ and write $\psi(A_2, B_2) = (C, D).$ Since $D = \mub_Y J$ for some $J \subseteq C,$ we have it suffices to show $D \subset \Greaters{Y}{c}$ and $c \in C$ to conclude $(c, \Greaters{X}{c}) \less{Y} (C, D).$ Having shown this, applying similar reasoning to $\psi^{-1}$ shows that $\vp^{-1}:\str Y \to \str X$ exists and is a poset map. Therefore, $\vp$ is an isomorphism, and by Theorem \ref{AB}, we get the desired result. 

If $b \in B_2,$ and $(K, b) \in \mathcal K_{\strXF}(a),$ then $\vp(K, b) = (J', \rho_2(b)) \in \mathcal K_{\strYF}(c).$ So $D = \rho_2(B_2) \subset \Greaters{Y}{c}.$ 

Let $K \subseteq A_2$ be a maximal subset of $A$ such that $\mub_X K = B_2.$ Note that $a \in K.$ Fix $b \in B_2,$ and let $K_1$ and $K_2$ be finite subsets of $X_1$ such that $K_1, K_2,$ and $K$ are pairwise disjoint and $$\mub_X K_1\cup K = \mub_X K_2 \cup K = \{b\}.$$ 

Now, $(K, B_2) <_{\strXF} (K_i \cup K, b)$ for $i = 1, 2,$ and if $(K', B_2) <_{\strXF} (K_i \cup K,b)$ for $i = 1, 2$ for some $K' \subset X_1,$ then $$K' \subseteq (K_1 \cup K) \cap (K_2 \cup K),$$ so $(K', B_2) \le_{\strXF} (K, B_2)$ because of how we chose $K_1, K_2$ and $K.$ In other words, $(K, B_2)$ is the greatest element of $(\strXF)_{B_2}$ that is less than both $(K\cup K_1, b)$ and $(K \cup K_2, b).$  Let $(J_i, \rho_2(b)) = \vp(K\cup K_i, b),$ and let $(S, \rho_2(B_2)) = \vp(K, \rho_2(B_2)).$ Since $(K \cup K_i, b) \in \mathcal K_{\strXF} (a),$ it follows that $c \in J_1 \cap J_2$ by the previous lemma. Moreover, since $(K, B_2) \le_{\strXF} (A, B_2),$ we have $(S, B_2) \le_{\strYF} (C, \rho_2(B_2)).$ We claim $c \in S.$

Both $(J_1, \rho_2(b))$ and $(J_2, \rho_2(b))$ both exceed $(S, \rho_2(B_2))$ in $\strYF.$ Moreover, since $\vp$ restricts to an isomorphism from the subposet $(\strXF)_{B_2}\cup(\strXF)_b$ onto $(\strYF)_{\rho_2(B_2)}\cup(\strYF)_{\rho_2(b)},$ the node $(S, \rho_2(B_2))$ is the greatest element of $(\strYF)_{\rho_2(B_2)}$ that is less than both $(J_1, \rho_2(b))$ and $(J_2, \rho_2(b)).$ Since $c \in J_1 \cap J_2,$ and $c<_Y \rho_2(B_2)$ by the previous paragraph, we have $(c, \rho_2(B_2)) <_{\strYF} (J_i, \rho_2(b)).$ So $(c, \rho_2(B_2)) <_{\strYF} (S, \rho_2(B_2)),$ and therefore $c\in S \subseteq C.$  \end{proof} 

\section{Acknowledgements} The author would like to thank his mentor, collaborator, and ``big sister" S. Loepp for many insightful comments and for greatly helping to improve the quality of the paper. The author also wishes to thank Washington \& Lee University for their support via the Lenfest grant.

\bibliographystyle{plain}  
\bibliography{DissertationBib} 
\index{Bibliography@\emph{Bibliography}}

\end{document}